\documentclass[english,11pt,leqno]{amsart}

\usepackage{amsmath,amsfonts,amssymb,amscd,amsthm,amsbsy,upref,mathrsfs}
\usepackage{color}
 \usepackage{comment}
 \usepackage{amsrefs}
 \usepackage[english]{babel}
 \usepackage{xspace}
 \usepackage[hmargin=2cm,vmargin=2.5cm]{geometry}

\newcommand{\N}{\mathbb{N}}
\newcommand{\R}{\mathbb{R}}
\newcommand{\K}{\mathbb{K}}

\newcommand{\Lin}{\mathscr{L}}

\newcommand{\liminfty}{\lim\limits_{n \to +\infty}}

\newcommand{\cof}{\operatorname{cof}}

\newcommand{\NA}{\operatorname{NA}}

\theoremstyle{plain}
\newtheorem{prop}{Proposition}
\newtheorem{theorem}[prop]{Theorem}
\newtheorem{lemma}[prop]{Lemma}
\newtheorem{corollary}[prop]{Corollary}

\newtheorem{thmAlfa}{Theorem}

\theoremstyle{definition}

\newtheorem{pb}{Problem}

\newtheorem*{thank}{Acknowledgments}
\theoremstyle{remark}

\begin{document}

\allowdisplaybreaks

\title[Norm attaining operators into locally AMUC Banach spaces]{Norm attaining operators into locally asymptotically midpoint uniformly convex Banach spaces}

\author{A.~Fovelle}
\address{Institute of Mathematics (IMAG) and Department of Mathematical Analysis, University of Granada, 18071, Granada, Spain}
\email{audrey.fovelle@ugr.es}

\thanks{Research partially supported by MCIN/AEI/10.13039/501100011033 grant PID2021-122126NB-C31 and by ``Maria de Maeztu'' Excellence Unit IMAG, reference CEX2020-001105-M funded by MCIN/AEI/10.13039/501100011033}


\begin{abstract} We prove that if $Y$ is a locally asymptotically midpoint uniformly convex Banach space which has either a normalized, symmetric basic sequence that is not equivalent to the unit vector basis in $\ell_1$, or a normalized sequence with upper p-estimates for some $p>1$, then $Y$ does not satisfy Lindenstrauss' property B.
\end{abstract}

\maketitle

 \setcounter{tocdepth}{1}

\section{Introduction}

The study of the denseness of norm-attaining operators started with the seminal paper by Bishop and Phelps \cite{BP}, in which they prove that every functional can be approximated by norm-attaining ones. In this same paper, they ask the following question: given $X$ and $Y$ two Banach spaces, does the set of norm-attaining operators from $X$ to $Y$, denoted by $\NA(X,Y)$ (that is $T \in \NA(X,Y)$ if $\|Tx\|=\|T\|$ for some $x \in B_X$, the unit ball of $X$) is dense in $\Lin(X,Y)$, the space of all (linear continous) bounded operators? This question was answered by the negative in 1963 by Lindenstrauss \cite{Lin}, who also gave some positive examples. Following \cite{Lin}, we say that a Banach space $Y$ has property B if $\NA(X,Y)$ is dense in $\Lin(X,Y)$ for every Banach space $X$. The negative example of Lindenstrauss was the following: any strictly convex space containing an isomorphic copy of $c_0$ fails property B. Thoroughly studied, important examples of Banach spaces having or failing property B have been given since (see for exemple \cite{Lin} \cite{Schachermayer} \cite{JW} \cite{Uhl} \cite{JW2} \cite{AAP}). \\
Let us mention that the first example of a reflexive space failing property B is due to Gowers in 1990 \cite{Go}, who proved that $\ell_p$ does not have property B when $p \in (1, \infty)$. \\

In \cite{Aguirre}, Aguirre proved that strictly convex spaces satisfying the extra condition of having either a normalized, symmetric basic sequence which is not equivalent to the unit vector basis in $\ell_1$, or a normalized sequence with upper p-estimates for some $p>1$ (see Section 2 for the definition), do not satisfy property B. This was extended to all strictly convex Banach spaces by Acosta \cite{AcostaSC}, who also proved that the same results holds for infinite-dimensional $L_1(\mu)$ spaces \cite{AcostaL1}. \\

The result of Aguirre was enough to deduce what he describes as the main result of his paper, namely: every infinite-dimensional uniformly convex Banach space fails property B. This is the result we will generalize in the asymptotic setting. To be more specific, we will prove the following theorem. \\

\begin{thmAlfa} \label{thm ppal}
Let $Y$ be a locally AMUC space which has either a normalized, symmetric basic sequence which is not equivalent to the unit vector basis in $\ell_1$, or a normalized sequence with upper $p$-estimates for some $1<p<\infty$. Then $Y$ fails property B.
\end{thmAlfa}

\section{Definitions and notation}

All Banach spaces in these notes are assumed to be real. We denote the closed unit ball of a Banach space $X$ by $B_X$, and its unit sphere by $S_X$. Given a Banach space $X$ with norm $\|\cdot\|_X$, we simply write $\|\cdot\|$ as long as it is clear from the context on which space it is defined. \\

First, let us define locally asymptotically midpoint uniformly convex spaces.

\subsection{Locally AMUC Banach spaces}

Let $Y$ be a Banach space. Let us denote by $\cof(Y)$ the set of all closed subspaces of $Y$ of finite codimension. For $y \in S_Y$ and $t \in \R^+$, let
\begin{align*}
\tilde{\delta}_Y(y,t)&=\sup_{E \in \cof(Y)} \inf_{z \in S_E} \max\{ \|y+tz\|, \|y-tz\|\}-1 \\
&=\sup_{E \in \cof(Y)} \inf_{\substack{z \in E \\ \|z\| \geq 1}} \max\{ \|y+tz\|, \|y-tz\|\}-1
\end{align*}
where the second equality follows from the fact that, for every $z \in Z$, the map $s \in (0, \infty) \mapsto \max \{ \|y+sz\|, \|y-sz\| \}$ is non-decreasing. \\
We say that $Y$ is \textit{locally asymptotically midpoint uniformly convex} (locally AMUC) if $\tilde{\delta}_Y(y,t)>0$ for every $y \in S_Y$ and every $t>0$. \\

\begin{prop}[Corollary 2.3 \cite{DKRRZ}] \label{prop DKRRZ}
If $Y$ is locally AMUC, then for every $y \in S_Y$ and every $t >0$, there exists $\delta>0$ such that 
\[ \limsup \max\{ \|y+tz_n\|, \|y-tz_n\| \} \geq 1+\delta \]
for every weakly null sequence $(z_n)_{n \in \N} \subset Y$ such that $\|z_n\| \geq 1$ for every $n \in \N$. Moreover, the converse holds if $Y$ does not contain $\ell_1$.
\end{prop}

We now introduce the family of spaces we will us as domain spaces for the counterexamples.

\subsection{A family of Banach spaces}

We say that a sequence $w=(w_n)_{n \in \N}$ of positive numbers is \textit{admissible} if it is decreasing, $w_1=1$ and $w \in c_0 \setminus \ell_1$. If $w$ is an admissible sequence, we can define an associated Banach space $d_*(w)$ as follows: let 
\[ d_*(w)=\Big\{ x=(x_n)_{n \in \N} \in c_0; \ \liminfty \frac{\sum_{k=1}^n \tilde{x}_k}{\sum_{k=1}^n w_k}=0  \Big\} \]
where $(\tilde{x}_n)_{n \in \N}$ is the decreasing rearrangement of $(|x_n|)_{n \in \N}$, endowed with the norm 
\[ \forall x \in d_*(w), \ \|x\|= \sup_{n \in \N} \frac{\sum_{k=1}^n \tilde{x}_k}{\sum_{k=1}^n w_k} . \]

The space $d_*(w)$ is known to be a predual of a Lorentz sequence space (see \cite{Garling}, \cite{Sargent}) and it has a symmetric basis $(e_n)_{n \in \N}$ that shares properties with the one of $c_0$. We will in particular use the following one, which proof can be found in \cite{JSP}, or \cite{Go} in the special case $w=(1/n)_{n \in \N}$.

\begin{lemma} \label{lemme pts extremaux}
For every $x \in S_{d_*(w)}$, we can find $m \in \N$ and $\delta \in (0,1)$ so that 
\[ \|x+\lambda e_n\| \leq 1 \]
for every $n \geq m$ and every $\lambda \in \K$ so that $|\lambda| \leq \delta$.
\end{lemma}

We finish this subsection with definitions about sequences  

\subsection{Symmetric basic sequences and upper $p$-estimates}

First of all, let us recall that a \textit{basic} sequence is an infinite sequence that is a basis of its closed linear span. If $(x_n)_{n \in \N}$ is a basis of a Banach space $X$, it is said to be \textit{symmetric} if every permutation $(x_{\sigma(n)})_{n \in \N}$ if $(x_n)_{n \in \N}$ is a basis of $X$, equivalent to the basis $(x_n)_{n \in \N}$. \\

Let us now recall the definition of having upper $p$-estimates, $1<p<\infty$, for a sequence. Let $X$ be a Banach space and $p \in (1, \infty)$. We say that a normalized sequence $(x_n)_{n \in \N}$ of elements of $X$ has \textit{upper $p$-estimates} if there exists a constant $C>0$ such that 
\[ \Big\| \sum_{k=1}^n a_k x_k \Big\| \leq C \Big( \sum_{k=1}^n |a_k|^p \Big)^{1/p} \]
for every $n \in \N$ and every sequence of scalars $(a_n)_{n \in \N} \subset \R$. \\

If we denote by $(f_n)_{n \in \N}$ the canonical vector basis of $\ell_p$, one can note that the existence of a normalized sequence $(x_n)_{n \in \N}$ in $X$ with upper $p$-estimates is equivalent to the existence of an operator $\alpha \in \Lin(\ell_p,X)$ satisfying $\|\alpha(f_n)\|=1$ for every $n \in \N$. \\

For a non-exhaustive list of Banach spaces with upper $p$-estimates, one can refer to \cite{Gonzalo}, \cite{DGJ}. 

\section{Results}

Theorem \ref{thm ppal} will be deduced from the following key proposition, which links norm-attaining operators and locally AMUC Banach spaces.

\begin{prop}
If $Y$ is a locally AMUC Banach space and $w$ is an admissible sequence, then every norm-attaining operator $T \in \NA(d_*(w),Y)$ from $d_*(w)$ into $Y$ satisfies $\lim \|Te_n\|=0$.
\end{prop}

\begin{proof}
Let $T \in \NA(d_*(w),Y)$. Without loss of generality, we can assume $\|T\|=1$. Then, there exists $x \in B_{d_*(w)}$ so that $\|Tx\|=\|T\|=1$ and by assumption on $Y$ and Proposition \ref{prop DKRRZ}, for every $t>0$ we can find $\delta(t)>0$ such that 
\[ \limsup \max\{ \|Tx+tz_n\|, \|Tx-tz_n\| \} \geq 1+\delta(t) \]
for every weakly null sequence $(z_n)_{n \in \N} \subset Y$ such that $\|z_n\| \geq 1$ for every $n \in \N$. \\
Moreover, by Lemma \ref{lemme pts extremaux}, we can find $\delta \in (0,1)$ and $m \in \N$ so that $x \pm \delta e_n \in B_{d_*(w)}$ for every $n \geq m$. \\
Assume now by contradiction that $\eta=\lim \|Te_n\|>0$ and let $\epsilon= \delta \big( \frac{\delta \eta}{2} \big) >0$. Up to extraction, we can assume without loss of generality that $\|Te_n\| \geq \frac{\eta}{2}$ for every $n \in \N$. Since the sequence $(Te_n)_{n \in \N}$ is weakly null, we get
\[ 1+\varepsilon \leq \limsup \max\{ \|Tx+\delta Te_n \|, \|Tx-\delta Te_n\| \} . \]
As $\|x \pm \delta e_n\| \leq 1$ for every $n \geq m$, the contradiction follows from
\[ 1+\varepsilon \leq \limsup \max \{ \|T(x+\delta e_n) \|, \|T(x-\delta e_n)\| \} \leq \|T\| \leq 1 . \qedhere \]
\end{proof}

Let us now describe how to deduce Theorem \ref{thm ppal} from the previous proposition. We will do it in two steps, starting with locally AMUC spaces that contain a normalized, symmetric basic sequence which is not equivalent to the unit vector basis in $\ell_1$. In \cite{Aguirre}, Aguirre proved the following result.

\begin{prop}[Proposition $4$ \cite{Aguirre}]
Let $Y$ be a Banach space containing a normalized, symmetric basic sequence $(y_n)_{n \in \N}$ which is not equivalent to the unit vector basis in $\ell_1$. Then there is an admissible sequence $w$ and an operator $T \in \Lin(d_*(w),Y)$ such that
\[ \forall n \in \N, \ Te_n=y_n . \]
\end{prop}

As a consequence, we immediately get

\begin{theorem}
Let $Y$ be a locally AMUC Banach space which has a normalized, symmetric basic sequence $(y_n)_{n \in \N}$ which is not equivalent to the unit vector basis in $\ell_1$. Then $Y$ fails property B.
\end{theorem}

From now on, let $w=\big( \frac{1}{n}\big)_{n \in \N}$, and $G=d_*(w)$, the space used by Gowers in \cite{Go} to prove that $\ell_p$ does not have property B for $1<p<\infty$. \\

Proposition $7$ from \cite{Aguirre} asserts that if $Y$ is a Banach space in which we can find a normalized sequence with upper $p$-estimates, then there exists a noncompact operator $T$ from $G$ into $Y$. By looking at the proof, one even gets the following proposition. For sake of completeness, we introduce a proof. 

\begin{prop}
Let $Y$ be a Banach space in which we can find a normalized sequence $(y_n) \subset Y$ with upper $p$-estimates. Then there exists a bounded linear operator $T$ from $G$ into $Y$ such that $\|Te_n\|=1$ for every $n \in \N$.
\end{prop}

\begin{proof}
From the hypothesis on the sequence $(y_n)_{n \in \N}$, we get a bounded linear operator $U \in \Lin(\ell_p, Y)$ sending the elements of the canonical basis of $\ell_p$ on the sequence $(y_n)_{n \in \N}$. Moreover, the space $G$ is known to be contained in $\ell_p$ and the formal identity defines a bounded linear operator $S \in \Lin(G, \ell_p)$ (see \cite{Go}, p.149). To conclude, just note that $T= US \in \Lin(G,Y)$ satisfies $\|Te_n\|=\|y_n\|=1$ for every $n \in \N$.
\end{proof}

The following result follows immediately.

\begin{theorem}
Let $Y$ be a locally AMUC space which has a normalized sequence with upper $p$-estimates, $1<p<\infty$. Then $Y$ fails property B.
\end{theorem}

Theorem \ref{thm ppal} gives us new examples of spaces without property B. Before mentioning such a class of new examples, let us remind the definition of $\ell_p$-sums. Let $(X_n)_{n \in \N}$ be a sequence of Banach spaces and $p \in [1, \infty)$. We define the sum $\left( \sum_{n \in \N} X_n \right)_{\ell_p}$ to be the space of sequences $(x_n)_{n \in \N}$, where $x_n \in X_n$ for all $n \in \N$, such that $\sum_{n \in \N} \|x_n\|_{X_n}^p$ is finite, and we set
\[ \|(x_n)_{n \in \N} \| = \Big( \sum_{n \in \N} \|x_n\|_{X_n}^p \Big)^{\frac{1}{p}} . \]
One can check that $\left( \sum_{n \in \N} X_n \right)_{\ell_p}$, endowed with the norm $\|\cdot\|$ defined above, is a Banach space. Now, we can state our last result.

\begin{corollary}
Let $(X_n)_{n \in \N}$ be a sequence of Banach spaces such that $X_n \neq \{0\}$ for every $n \in \N$. Then any infinite dimensional subspace of $\big( \sum_{n \in \N} X_n \big)_{\ell_p}$ fails property B.
\end{corollary}

\section{Open questions}

\begin{pb}
Does Theorem \ref{thm ppal} hold for all locally AMUC Banach spaces?
\end{pb}

One could start by the answering the following intermediate question.

\begin{pb}
If $(F_n)$ is a sequence of finite-dimensional spaces ($\dim(F_n)>0$ for every $n \in \N$), does $\big( \sum_{n \in \N} F_n \big)_{\ell_1}$ fail property B?
\end{pb}

We will finish this paper by recalling the following longstanding open problem, even for the $2$-dimensional Euclidean space.

\begin{pb}
Do finite dimensional spaces have property B?
\end{pb}

\begin{thank}
The author would like to thank Miguel Martín for very useful conversations and valuable comments.
\end{thank}

\bibliographystyle{plain}
\bibliography{biblio}

\end{document}